\documentclass[12pt,reqno]{amsart}
\newtheorem{theorem}{Theorem}[section]
\newtheorem{definition}{Definition}[section]

\theoremstyle{definition}
\newtheorem{remark}{Remark}[section]
\newtheorem{example}{Example}[section]

\newcommand{\be}{\begin{equation}}
\newcommand{\ee}{\end{equation}}
\newcommand{\bea}{\begin{eqnarray}}
\newcommand{\eea}{\end{eqnarray}}
\newcommand{\beb}{\begin{eqnarray*}}
\newcommand{\eeb}{\end{eqnarray*}}
\newcommand{\norm}[1]{\left\lVert#1\right\rVert}
\usepackage{amssymb,amsfonts,amsthm,setspace,indentfirst,mathrsfs}
\usepackage{minibox}
\usepackage{enumitem}
\usepackage{tikz-cd}
\usepackage{centernot}
\usepackage{stmaryrd}
\usepackage{hyperref}
\numberwithin{equation}{section}
\begin{document}
\title[$r-st_2^{\vartheta}$-convergence in PNS]{Rough statistical convergence of double  sequences in probabilistic normed spaces}

\author[ R. Mondal, N. Hossain]{$^1$Rahul Mondal, $^{2}$Nesar Hossain}

\address{$^{1}$Department of Mathematics, Vivekananda Satavarshiki Mahavidyalaya, Manikpara, Jhargram -721513, West Bengal, India.} 
\address{$^{2}$Department of Mathematics, The University of Burdwan,  Burdwan - 713104, West Bengal, India.}

\email{imondalrahul@gmail.com$^{1}$; nesarhossain24@gmail.com$^{2}$}

%\date{\today}

\begin{abstract}
In this paper, we have defined rough convergence and rough statistical convergence of double sequences in probabilistic normed spaces which is more generalized version than the rough statistical convergence of double sequences in normed linear spaces. Also, we have defined rough statistical cluster points of double sequences and then, investigated some important results associated with the set of rough statistical limits of double sequences in these spaces. Moreover, in the same spaces, we have proved an important relation between the set of all rough statistical cluster points and rough statistical limits under certain condition.
\end{abstract}

\maketitle

\section{Introduction}
In $1951$, the concept of usual convergence of real sequences was extended to statistical convergence of real sequences based on the natural density of a set by Fast \cite{Fast} and Steinhaus \cite{Steinhaus} independently. Later on, this idea has been studied in different directions and in different spaces by many authors as in \cite{Connor, Cakalli, Dundar, Fridy, Fridy1993, Mursaleen2000,   Mursaleen2003,  Nuray,  Salat, Sarabadan} and many more.

In $2001$, Phu \cite{Phu2001} has initially introduced the concept of rough convergence of sequences in finite dimensional normed linear spaces which is basically a generalization of usual convergence and, in the same paper he has investigated that $r$-limit set is bounded, closed, convex and many more interesting results and later on, this concept has been extended to infinite dimensional normed linear spaces \cite{Phu2003}. Also, He \cite{Phu2002} has defined the notion of rough continuity of linear operators. Later, Ayter \cite{Ayter} extended this notion to rough statistical convergence based on natural density of a set. Malik and Maity \cite{Malik2013, Malik2016} has defined rough convergence and rough statistical convergence of double sequences in normed linear spaces. After that, the research work on this concept is still being carried out in different directions as in \cite{Antal2021, Ghosal, Hossain, Kisi, Ozcan} and many references therein.

In $1942$, Menger \cite{Menger} first proposed the concept of statistical metric space, now called probabilistic metric space, which is an interesting and important generalization of the notion of metric space. This concept, later on, was studied by Schweizer and Sklar \cite{Schweizer}. Combining the idea of statistical metric space and normed linear space,  \v{S}erstnev \cite{Serstnev} introduced the idea of probabilistic normed space. In $1993$ Alsina et al. gave a new definition of probabilistic normed space whic is basically a special case of the definition of \v{Serstnev}. Recently,  Antal et al. \cite{Antal 2022} defined the notion of rough convergence and rough statistical convergence in probabilistic normed spaces. In this space, we have presented the notion of rough statistical convergence of double sequences and investigated some interesting results associated with the sets of rough statistical cluster points and rough statistical limits of double sequences.
\section{Preliminaries}
Throughout the paper $\mathbb{N}$ and $\mathbb{R}$ denote the set of positive integers and set of reals respectively. First we recall some basic definitions and notations.

\begin{definition}\cite{Schweizer}
  A triangular norm, briefly $t$-norm, is a binary operation on $[0,1]$ which is continuous, commutative, associative, non decreasing and has $1$ as unit element, i.e., it is the continuous mapping $\diamond: [0,1]\times [0,1]\rightarrow [0,1]$  such that for all $a,b,c,d\in[0,1]$:
  \begin{enumerate}
      \item $a\diamond 1=a$;
      \item $a\diamond b=b\diamond a$;
      \item $a\diamond b\geq c\diamond d$ whenever $a\geq c$ and $b\geq d$;
      \item $a\diamond (b\diamond c)=(a\diamond b)\diamond c $.
  \end{enumerate}
\end{definition}

\begin{example}\cite{Klement}
  The following are the examples of $t$-norms:
  \begin{enumerate}
      \item $x\diamond y= min\{x,y\}$;
      \item $x\diamond y=x.y$;
      \item $x\diamond y= max\{x+y-1,0\}$. This $t$-norm is known as Lukasiewicz $t$-norm.
  \end{enumerate}
\end{example}

\begin{definition}\cite{Frank}
  A function $f: \mathbb{R}\rightarrow \mathbb{R}_0^+$  is said to be a distribution function if it is non decreasing and left continuous with $\inf_{t\in\mathbb{R}} f(t)=0$ and $\sup_{t\in\mathbb{R}} f(t)=1$. We denote $D$ as the set of all distribution functions.
\end{definition}

\begin{definition}\cite{Frank}
    A triplet $(X,\vartheta,\diamond)$ is called a probabilistic normed space (shortly PNS) if $X$ is a real vector space, $\nu$ is a mapping from $X$ into $D$ (for $x\in X, t\in \mathbb(R)$, $\vartheta(x;t)$ is the value of the distribution function $\vartheta(x)$ at $t$) and $\diamond$ is a $t$-norm satisfying the following conditions:
    \begin{enumerate}
        \item $\vartheta(x;0)=0$;
        \item $\vartheta(x;t)=1$, $\forall \ t>0$ iff $x=\theta$, $\theta$ being the zero element of $X$;
        \item $\vartheta(\alpha x;t)=\vartheta(x;\frac{t}{|\alpha|})$, $\forall \ \alpha\in\mathbb{R}\setminus \{0\}$ and $\forall \ t>0$;
        \item $\vartheta(x+y;s+t)\geq \vartheta(x;t)\diamond\vartheta(y;s)$, $\forall \ x,y\in X$ and $\forall\ s,t\in \mathbb{R}_0^+$.
    \end{enumerate}
\end{definition}

\begin{example}\cite{Aghajani}
    For a real normed space $(X,\norm{\cdot})$, we define the probabilistic norm $\vartheta$ for $x\in X, t\in \mathbb{R}$ as $\vartheta(x;t)=\frac{t}{t+\norm{x}}$. Then $(X,\vartheta,\diamond)$ is a PNS under the $t$-norm $\diamond$ defined by $x\diamond y=\min \{x,y\}$. Also, $x_n\xrightarrow{\norm{\cdot}}\xi$ if and only if $x_n\xrightarrow{\vartheta}\xi$.
\end{example}

\begin{definition}\cite{Aghajani}
    Let $(X,\vartheta,\diamond)$ be a PNS. For $r>0$, the open ball $B(x,\lambda;r)$ with center $x\in X$ and radius $\lambda\in(0,1)$ is the set $$B(x,\lambda;r)=\{y\in X: \vartheta(y-x;r)>1-\lambda\}.$$ 
    Similarly, the  closed ball is the set $\overline{B(x,\lambda;r)}=\{y\in X: \vartheta(y-x;r)\geq 1-\lambda\}$
\end{definition}

\begin{definition}\cite{Karakus Demirci}
    Let $\{x_{mn}\}$ be a double sequence in a PNS $(X,\vartheta,\diamond)$. Then $\{x_{mn}\}$ is said to be convergent to $\xi\in X$ with respect to the probabilistic norm $\vartheta$ if for every $\varepsilon>0$ and $\lambda\in(0,1)$, there exists a positive integer $n_0$ such that $\vartheta(x_{mn}-\xi;\varepsilon)>1-\lambda$ whenever $m,n\geq n_0$. In this case we write $\vartheta_2\text{-}\lim x_{mn}=\xi$ or $x_{mn}\xrightarrow{\vartheta_2}\xi$.
\end{definition}

\begin{definition}
Let $K\subset \mathbb{N}$. Then the natural density $\delta(K)$ of $K$  is defined by $$\delta(K)=\lim_{n\to\infty}\frac{1}{n}|\{k\leq n: k\in K\}|,$$ provided the limit exists. 
\end{definition}
It is clear that if $K$ is finite then $\delta (K)=0$.

\begin{definition}\cite{Antal 2022}
    Let $\{x_n\}_{n\in\mathbb{N}}$ be a sequence in an PNS $(X,\vartheta,\diamond)$. Then   $\{x_n\}_{n\in\mathbb{N}}$ is said to be rough convergent to $\xi\in X$ with respect to the probabilistic norm $\vartheta$ if for every $\varepsilon>0$, $\lambda\in (0,1)$ and some non negative number $r$ there exists $n_0\in\mathbb{N}$ such that $\vartheta(x_n-\xi; r+\varepsilon)>1-\lambda$ for all $n>n_0$. In this case we write $r_\vartheta\text{-}\lim_{n\to\infty}x_n=\xi\ \text{or}\ x_n\xrightarrow{r_{\vartheta}}\xi$ and $\xi$ is called $r_\vartheta$-limit of $\{x_n\}_{n\in\mathbb{N}}$.
\end{definition}

\begin{definition}\cite{Antal 2022}
    Let $\{x_n\}_{n\in\mathbb{N}}$ be a sequence in an PNS $(X,\vartheta,\star)$.  Then $\{x_n\}_{n\in\mathbb{N}}$ is said to be rough statistically convergent to $\xi\in X$ with respect to the probabilistic norm $\vartheta$ if for every $\varepsilon>0$ and $\lambda\in(0,1)$ and some non negative number $r$, $\delta(\{n\in\mathbb{N}:\vartheta(x_n-\xi; r+\varepsilon)\leq 1-\lambda\})=0$. In this case we write $r\text{-}St_\vartheta\text{-}\lim_{n\to\infty}x_n=\xi\ \text{or}\ x_n\xrightarrow{r-St_\vartheta}\xi$.
\end{definition}

\begin{definition}\cite{Mursaleen2003}
    The double natural density of the set $K\subseteq \mathbb{N}\times\mathbb{N}$ is defined by $$\delta_2(K)=\lim_{m,n\to\infty}\frac{|\{(i,j)\in K: i\leq m\ \text{and}\ j\leq n \}|}{mn}$$ where $|\{(i,j)\in K: i\leq m\ \text{and}\ j\leq n \}|$ denotes the number of elements of $K$ not exceeding $m$ and $n$ respectively. It can be observed that if $K$ is finite, then $\delta_2(K)=0$. Also, if $A\subseteq B$, then $\delta_2(A)\leq \delta_2(B)$.
\end{definition}

\begin{definition}\cite{Karakus Demirci}
    Let $\{x_{mn}\}$ be a double sequence in a PNS $(X,\vartheta,\diamond)$. Then $\{x_{mn}\}$ is said to be statistically convergent to $\xi\in X$ with respect to the probabilistic norm $\vartheta$ if for every $\varepsilon>0$ and $\lambda\in(0,1)$, $K=\{ (m,n), m\leq i, n\leq j: \vartheta(x_{mn}-\xi;\varepsilon)\leq 1-\lambda\}$ has double natural density zero, that is, if $K(i,j)$ become the numbers of $(m,n)$ in $K$: $$\lim_{i,j}\frac{K(i,j)}{ij}=0.$$ In this case we write $st_2^\vartheta\text{-}\lim x_{mn}=\xi$ or $x_{mn}\xrightarrow{st_2^\vartheta}\xi$.
\end{definition}

\begin{definition}\cite{Malik2016}
    A  subsequence $x'=\{x_{j_p k_q}\}$ of a double sequence $\{x_{jk}\}$ is called a dense subsequence , if $\delta_2(\{(j_p k_q )\in\mathbb{N}\times\mathbb{N}: p,q\in\mathbb{N}\})=1$.
\end{definition}

\section{Main Results}

First we define rough convergence and rough statistical convergence of double sequences in probabilistic normed spaces.

\begin{definition}\label{defi3.1}
Let $\{x_{mn}\}$ be a double sequence in a PNS $(X,\vartheta,\diamond)$ and $r$ be a non negative real number. Then $\{x_{mn}\}$ is said to be rough convergent to $\beta\in X$ with respect to the probabilistic norm $\vartheta$ if for every $\varepsilon>0$, $\lambda\in(0,1)$ there exists $n_0\in\mathbb{N}$ such that $\vartheta(x_{mn}-\beta;r+\varepsilon)>1-\lambda$ for all $m,n\geq n_0$. In this case $\beta$ is called $r_2^\vartheta$-limit of $\{x_{mn}\}$ and we write $x_{mn}\xrightarrow{r_2^\vartheta}\beta$.
\end{definition}

\begin{definition}\label{defi3.2}
    Let $\{x_{mn}\}$ be a double sequence in a PNS $(X,\vartheta,\diamond)$ and $r$ be a non negative real number. Then $\{x_{mn}\}$ is said to be rough statistical  convergent to $\beta\in X$ with respect to the probabilistic norm $\vartheta$ if for every $\varepsilon>0$, $\lambda\in(0,1)$, $\delta_2(\{(m,n)\in\mathbb{N}\times\mathbb{N}: \vartheta(x_{mn}-\beta;r+\varepsilon) \leq 1-\lambda\})=0$. In this case $\beta$ is called $r\text{-}st_2^\vartheta$-limit of $\{x_{mn}\}$ and we write $r\text{-}st_2^\vartheta\text{-}\lim_{m,n\to\infty}x_{mn}=\beta$ or $x_{mn}\xrightarrow{r-st_2^\vartheta}\beta$.
\end{definition}

\begin{remark}
$(a)$ If we put $r=0$ in Definition \ref{defi3.1}, then the notion of rough convergence of a double sequence with respect to the probabilistic norm $\vartheta$ coincides with notion of ordinary convergence of the double sequence  with respect to the probabilistic norm $\vartheta$.\\
$(b)$ From Definition \ref{defi3.1}, it is clear that $r_2^\vartheta$-limit of a double sequence may not be unique. So, we denote $LIM_{x_{mn}}^{r_\vartheta}$ to mean the set of all $r_2^\vartheta$-limit of $\{x_{mn}\}$ with respect to the probabilistic norm $\vartheta$.\\
$(c)$ If we put $r=0$ in Definition \ref{defi3.2}, then the notion of rough statistical convergence of a double sequence with respect to the probabilistic norm $\vartheta$ coincides with statistical convergence of the double sequence  with respect to the probabilistic norm $\vartheta$. So, our whole discussion is on the fact $r>0$. \\
$(d)$ From Definition \ref{defi3.2}, it is clear that $r\text{-}st_2^\vartheta$-limit of a double sequence may not be unique. So, we denote $st_2^\vartheta\text{-}LIM_{x_{mn}}^r$ to mean the set of all $r\text{-}st_2^\vartheta$-limit of $\{x_{mn}\}$ with respect to the probabilistic norm $\vartheta$.
\end{remark}

The sequence $\{x_{mn}\}$ is said to be $r_2^\vartheta$-convergent if $LIM_{x_{mn}}^{r_\vartheta}\neq \emptyset$. But, if the sequence is unbounded with respect to the probabilistic norm $\vartheta$ then $LIM_{x_{mn}}^{r_\vartheta}= \emptyset$ although in this case  $st_2^\vartheta\text{-}LIM_{x_{mn}}^r\neq \emptyset$ may be happened which has been shown in the following example.

\begin{example}\label{exmp3.1}
Let $(X,\norm{\cdot})$ be a real normed linear space and let $\vartheta(x;t)=\frac{t}{t+\norm{x}}$ for $x\in X$ and $t>0$. Then $(X,\vartheta,\diamond)$ is a PNS under the $t$-norm $\diamond$ defined by $x\diamond y=\min\{x,y\}$. For all $m,n\in \mathbb{N}$, we define a sequence $\{x_{mn}\}$ by $x_{mn}=\begin{cases}
(-1)^{m+n}, \ m,n\neq i^2 \ (i\in\mathbb{N})\\
mn, \ \text{otherwise}
\end{cases}$. Then, we have $st_2^\vartheta\text{-}LIM_{x_{mn}}^r=\begin{cases}
    \emptyset , \ r<1 \\
    [1-r,r-1], \ \text{otherwise}
\end{cases}$ and $st_2^\vartheta\text{-}LIM_{x_{mn}}^r=\emptyset$ when $r=0$. Also, $LIM_{x_{mn}}^{r_\vartheta}=\emptyset$ for any $r\geq 0$.
\end{example}

\begin{remark}
    From Example \ref{exmp3.1}, we have $st_2^\vartheta\text{-}LIM_{x_{mn}}^r\neq \emptyset$ does not imply $LIM_{x_{mn}}^{r_\vartheta}\neq \emptyset$. But, $LIM_{x_{mn}}^{r_\vartheta}\neq \emptyset$ always implies $LIM_{x_{mn}}^{r_\vartheta}\neq \emptyset$ as $\delta_2(\{(m,n)\in \mathbb{N}\times\mathbb{N}: \ \text{either}\ m \ \text{or}\ n \ \text{runs over finite subsets of}\ \mathbb{N} \})=0$. So, $LIM_{x_{mn}}^{r_\vartheta}\subset st_2^\vartheta\text{-}LIM_{x_{mn}}^r$.
\end{remark}

\begin{example}\label{exmp3.2}
    We take the PNS in Example \ref{exmp3.1} and define the double sequence  $\{x_{mn}\}$ by  $x_{mn}=\begin{cases}
        mn, \ m,n=i^2\ (i\in\mathbb{N})\\
        0, \ \text{otherwise}
    \end{cases}$. Then, $st_2^\vartheta\text{-}LIM_{x_{mn}}^r=[-r,r]$. Now, if we consider a subsequence $\{x_{m_j n_k}\}$ of $\{x_{mn}\}$ such that $m_j=j^2,n_k=k^2$, $j,k\in\mathbb{N}$, then $st_2^\vartheta\text{-}LIM_{x_{m_j n_k}}^r=\emptyset$.
\end{example}

\begin{remark}
  From Example \ref{exmp3.2}, for any subsequence of a double sequence we not not conclude that $st_2^\vartheta\text{-}LIM_{x_{mn}}^r\subseteq st_2^\vartheta\text{-}LIM_{x_{m_j n_k}}^r$.
\end{remark}

But, this inclusion may be hold under certain condition which has been given in the following theorem.

\begin{theorem}
Let $\{x_{m_j n_k}\}$ be a dense subsequence of $\{x_{mn}\}$ in a PNS $(X,\vartheta,\diamond)$. Then $st_2^\vartheta\text{-}LIM_{x_{mn}}^r\subseteq st_2^\vartheta\text{-}LIM_{x_{m_j n_k}}^r$.
\end{theorem}

\begin{proof}
The proof is obvious. So, we omit details.
\end{proof}

\begin{definition}
   Let $\{x_{mn}\}$ be a double sequence in a PNS $(X,\vartheta,\diamond)$. Then $\{x_{mn}\}$ is said to be statistically bounded with respect to the probabilistic norm $\vartheta$ if for every $\lambda\in(0,1)$ there exists a positive real number $G$ such that $\delta_2(\{(m,n)\in\mathbb{N}\times \mathbb{N}: \vartheta(x_{mn};G)\leq 1-\lambda  \})=0$.
\end{definition}

\begin{theorem}
   Let $\{x_{mn}\}$ be a double sequence in a PNS $(X,\vartheta,\diamond)$. Then $\{x_{mn}\}$ is statistically bounded if and only if  $st_2^\vartheta\text{-}LIM_{x_{mn}}^r\neq \emptyset$ for some $r>0$.
\end{theorem}

\begin{proof}
    First suppose that $\{x_{mn}\}$ is statistically bounded. Then for every $\lambda\in(0,1)$ there exists a positive real number $G$ such that $\delta_2(\{(m,n)\in\mathbb{N}\times \mathbb{N}: \vartheta(x_{mn};G)\leq 1-\lambda  \})=0$. Now, let $M=\{(m,n)\in\mathbb{N}\times \mathbb{N}: \vartheta(x_{mn};G)\leq 1-\lambda  \}$ and $\theta$ be the zero element in $X$. Now for $m,n\in M^c$ we have $\vartheta(x_{mn}-\theta;r+G)\geq \vartheta(x_{mn};G)\diamond \vartheta(\theta;r)>(1-\lambda)\diamond 1=1-\lambda$. This gives $\theta\in st_2^\vartheta\text{-}LIM_{x_{mn}}^r$ and consequently $st_2^\vartheta\text{-}LIM_{x_{mn}}^r\neq \emptyset$.

    Conversely suppose that $st_2^\vartheta\text{-}LIM_{x_{mn}}^r\neq \emptyset$. Let $\xi\in st_2^\vartheta\text{-}LIM_{x_{mn}}^r\neq \emptyset$. Then for every $\varepsilon>0$ and $\lambda\in(0,1)$, $\delta_2(\{(m,n)\in\mathbb{N}\times\mathbb{N}: \vartheta(x_{mn}-\xi;r+\varepsilon)\leq 1-\lambda\})=0$. Therefore almost all $x_{mn}$ are contained in some ball with center $\xi$. This shows that $\{x_{mn}\}$  is statistically bounded. This completes the proof.
\end{proof}

Now we give the algebraic characterization of rough statistically convergent double sequences in probabilistic normed spaces.

\begin{theorem}
    Let $\{x_{mn}\}$ and $\{y_{mn}\}$ be  double sequences in a PNS $(X,\vartheta,\diamond)$. Then for some $r>0$ the following statements hold:
\begin{enumerate}
    \item If $x_{mn}\xrightarrow{r-st_2^\vartheta}\beta$ and $y_{mn}\xrightarrow{r-st_2^\vartheta}\eta$ then $x_{mn}+ y_{mn}\xrightarrow{r-st_2^\vartheta}\beta +\eta$
    \item If $x_{mn}\xrightarrow{r-st_2^\vartheta}\beta$ and $\alpha (\neq 0)\in\mathbb{R}$ then $\alpha x_{mn}\xrightarrow{r-st_2^\vartheta}\alpha \beta$.
\end{enumerate}
\end{theorem}

\begin{proof}
    Let $\{x_{mn}\}$ and $\{y_{mn}\}$ be  double sequences in  PNS $(X,\vartheta,\diamond)$ and $r>0$.
    \begin{enumerate}
        \item Let $x_{mn}\xrightarrow{r-st_2^\vartheta}\beta$ and $y_{mn}\xrightarrow{r-st_2^\vartheta}\eta$. Let $\varepsilon>0$. Now, for a given $\lambda\in(0,1)$, choose $s\in(0,1)$ such that $(1-s)\diamond(1-s)>1-\lambda$. So, $\delta_2(A)=0$ and $\delta_2(B)=0$ where $A=\{(m,n)\in\mathbb{N}\times\mathbb{N}: \vartheta(x_{mn}-\beta;\frac{r+\varepsilon}{2})\leq 1-s \}$ and $B=\{(m,n)\in\mathbb{N}\times\mathbb{N}: \vartheta(y_{mn}-\eta;\frac{r+\varepsilon}{2})\leq 1-s \}$. Now for $(i,j)\in A^c\cap B^c$, we have $\vartheta(x_{ij}+y_{ij}-(\beta+\eta);r+\varepsilon)\geq \vartheta(x_{ij}-\beta;\frac{r+\varepsilon}{2})\diamond(y_{ij}-\eta;\frac{r+\varepsilon}{2})>(1-s)\diamond (1-s)>1-\lambda$, i.e. $A^c\cap B^c\subset \{ (i,j)\in\mathbb{N}\times\mathbb{N}: \vartheta(x_{ij}+y_{ij}-(\beta+\eta);r+\varepsilon)> 1-\lambda \}$. Therefore $\delta_2(\{ (i,j)\in\mathbb{N}\times\mathbb{N}: \vartheta(x_{ij}+y_{ij}-(\beta+\eta);r+\varepsilon)\leq  1-\lambda \})=0$, which gives $x_{mn}+ y_{mn}\xrightarrow{r-st_2^\vartheta}\beta +\eta$.
        \item Since $x_{mn}\xrightarrow{r-st_2^\vartheta}\beta$ and $\alpha\neq 0$, then for every $\varepsilon>0$ and $\lambda\in(0,1)$, $\delta_2(\{(m,n)\in\mathbb{N}\times \mathbb{N}: \vartheta(x_{mn}-\beta;\frac{r+\varepsilon}{|\alpha|})\leq 1-\lambda\})=0$ i.e., $\delta_2(\{(m,n)\in\mathbb{N}\times \mathbb{N}: \vartheta(\alpha x_{mn}-\alpha \beta; r+\varepsilon)\leq 1-\lambda\})=0$, which gives $\alpha x_{mn}\xrightarrow{r-st_2^\vartheta}\alpha \beta$. This completes the proof.
    \end{enumerate}
\end{proof}

We will discuss on some topological and geometrical properties of the set $st_2^\vartheta\text{-}LIM_{x_{mn}}^r$.

\begin{theorem}
     Let $\{x_{mn}\}$ be a double sequence in a PNS $(X,\vartheta,\diamond)$. Then the set $st_2^\vartheta\text{-}LIM_{x_{mn}}^r$ is closed.
\end{theorem}

\begin{proof}
    If $st_2^\vartheta\text{-}LIM_{x_{mn}}^r=\emptyset$ then we have nothing to prove. So, let $st_2^\vartheta\text{-}LIM_{x_{mn}}^r\neq \emptyset$. Suppose that $\{y_{mn}\}$ is a double sequence in $st_2^\vartheta\text{-}LIM_{x_{mn}}^r$ such that $y_{mn}\xrightarrow{\vartheta_2}\beta$. For a given $\lambda\in(0,1)$ choose $s\in(0,1)$ such that $(1-s)\diamond (1-s)>1-\lambda$. Then for every $\varepsilon>0$ there exists $n_0\in\mathbb{N}$, such that $\vartheta(y_{mn}-\beta; \frac{\varepsilon}{2})>1-s$ for all $m,n>n_0$. Suppose $i,j>n_0$. Then $\vartheta(y_{ij}-\beta; \frac{\varepsilon}{2})>1-s$. Again, since $\{y_{ij}\}\in st_2^\vartheta\text{-}LIM_{x_{mn}}^r$, then $\delta_2(P)=0$ where $P=\{(m,n)\in\mathbb{N}\times\mathbb{N}: \vartheta(x_{mn}-y_{ij}; r+\frac{\varepsilon}{2})\leq 1-s \}$. Now for $(s,t)\in P^c$, we have $\vartheta(x_{st}-\beta;r+\varepsilon)\geq \vartheta(x_{st}-y_{ij};r+\frac{\varepsilon}{2})\diamond \vartheta(y_{ij}-\beta;\frac{\varepsilon}{2})>(1-s)\diamond (1-s)>1-\lambda$. Therefore $\{ (s,t)\in\mathbb(N)\times\mathbb{N}: \vartheta(x_{st}-\beta;r+\varepsilon)\leq 1-\lambda\}\subset P$. Since $\delta_2(P)=0$, therefore $\delta_2(\{ (s,t)\in\mathbb(N)\times\mathbb{N}: \vartheta(x_{st}-\beta;r+\varepsilon)\leq 1-\lambda\})=0$. Consequently $\beta\in st_2^\vartheta\text{-}LIM_{x_{mn}}^r$. So, $st_2^\vartheta\text{-}LIM_{x_{mn}}^r$ is closed. This completes the proof.
\end{proof}

\begin{theorem}
    Let $\{x_{mn}\}$ be a double sequence in a PNS $(X,\vartheta,\diamond)$. Then the set $st_2^\vartheta\text{-}LIM_{x_{mn}}^r$ is convex for some $r>0$. 
\end{theorem}

\begin{proof}
    Let $\xi_1,\xi_2\in st_2^\vartheta\text{-}LIM_{x_{mn}}^r$ and $t\in (0,1)$. Suppose $\lambda\in(0,1)$. Choose $s\in(0,1)$ such that $(1-s)\diamond (1-s)>1-\lambda$. Then $\delta_2(A)=0$ and $\delta_2(B)=0$ where $A=\{ (m,n)\in\mathbb{N}\times\mathbb{N}: \vartheta(x_{mn}-\xi_1;\frac{r+\varepsilon}{2(1-t)})\leq 1-s \}$ and $B= \{ (m,n)\in\mathbb{N}\times\mathbb{N}: \vartheta(x_{mn}-\xi_2;\frac{r+\varepsilon}{2t})\leq 1-s \}$. Now for $i,j\in A^c\cap B^c$ we have $\vartheta(x_{ij}-[(1-t)\xi_1+t\xi_2];r+\varepsilon)\geq \vartheta((1-t)(x_{ij}-\xi_1);\frac{r+\varepsilon}{2})\diamond \vartheta(t(x_{ij}-\xi_2);\frac{r+\varepsilon}{2})=\vartheta(x_{ij}-\xi_1;\frac{r+\varepsilon}{2(1-t)})\diamond \vartheta(x_{ij}-\xi_2;\frac{r+\varepsilon}{2t})>(1-s)\diamond(1-s)>1-\lambda$. Therefore $\{(i,j)\in \mathbb{N}\times{N}: \vartheta(x_{ij}-[(1-t)\xi_1+t\xi_2];r+\varepsilon)\leq 1-\lambda\}\subset A\cup B$. Hence $\delta_2(\{(i,j)\in \mathbb{N}\times \mathbb{N}: \vartheta(x_{ij}-[(1-t)\xi_1+t\xi_2];r+\varepsilon)\leq 1-\lambda\})=0$, i.e. $(1-t)\xi_1+t\xi_2\in st_2^\vartheta\text{-}LIM_{x_{mn}}^r$. Therefore $st_2^\vartheta\text{-}LIM_{x_{mn}}^r$ is convex. This completes the proof.
\end{proof}

\begin{theorem}
    A double sequence $\{x_{mn}\}$ in a PNS $(X,\vartheta,\diamond)$ is rough statistically convergent to $\xi\in X$ with respect to the probabilistic norm $\vartheta$ for some $r>0$ if there exists  a double sequence $\{y_{mn}\}$ in $X$ such that $st_2^\vartheta\text{-}\lim y_{mn}=\xi$ and for every $\lambda\in(0,1)$, $\vartheta(x_{mn}-y_{mn};r)>1-\lambda$ for all $m,n\in\mathbb{N}$.
\end{theorem}

\begin{proof}
    Let $\varepsilon>0$ be given. For a given $\lambda\in(0,1)$, choose $s\in(0,1)$ such that $(1-s)\diamond(1-s)>1-\lambda$. Suppose that $st_2^\vartheta\text{-}\lim y_{mn}=\xi$ and $\vartheta(x_{mn}-y_{mn};r)>1-\lambda$ for all $m,n\in\mathbb{N}$. Then $\delta_2(A)=0$ where $A=\{(m,n)\in\mathbb{N}\times\mathbb{N}: \vartheta(y_{mn}-\xi;\varepsilon)\leq 1-s\}$. Now for $(i,j)\in A^c$, we have $\vartheta(x_{ij}-\xi;r+\varepsilon)\geq \vartheta(x_{ij}-y_{ij};r)\diamond \vartheta(y_{ij}-\xi;r)>(1-s)\diamond(1-s)>1-\lambda$. Therefore $\{(i,j)\in\mathbb{N}\times\mathbb{N}: \vartheta(x_{ij}-\xi;r+\varepsilon)\leq 1-\lambda\}\subset A$. Hence $\delta_2(\{(i,j)\in\mathbb{N}\times\mathbb{N}: \vartheta(x_{ij}-\xi;r+\varepsilon)\leq 1-\lambda\})=0$. Consequently $x_{mn}\xrightarrow{r-st_2^\vartheta}\xi$. This completes the proof.
\end{proof}

\begin{theorem}
   Let  $\{x_{mn}\}$ be a double sequence in a PNS $(X,\vartheta,\diamond)$. Then there do not exist $x_1,x_2\in st_2^\vartheta\text{-}LIM_{x_{mn}}^r$ for some $r>0$ and every $\lambda\in (0,1)$ such that $\vartheta(x_1-x_2;mr)\leq 1-\lambda$ for $m(\in\mathbb{R})>2$.
\end{theorem}

\begin{proof}
    On contrary, we assume that there exist $x_1,x_2\in st_2^\vartheta\text{-}LIM_{x_{mn}}^r$  for which $\vartheta(x_1-x_2;mr)\leq 1-\lambda$ for $m(\in\mathbb{R})>2$. Let $\varepsilon>0$ be given. For a given $\lambda\in(0,1)$, choose $s\in(0,1)$ such that $(1-s)\diamond (1-s)>1-\lambda$. Then $P=\{(m,n)\in\mathbb{N}\times\mathbb{N}: \vartheta(x_{mn}-x_1;r+\frac{\varepsilon}{2})\leq 1-s\}$ and $Q=\{(m,n)\in\mathbb{N}\times\mathbb{N}: \vartheta(x_{mn}-x_2;r+\frac{\varepsilon}{2})\leq 1-s\}$ have double natural density zero. Now for $(i,j)\in P^c\cap Q^c$, we have $\vartheta(x_1-x_2;2r+\varepsilon)\geq \vartheta(x_{ij}-x_1;r+\frac{\varepsilon}{2})\diamond \vartheta(x_{ij}-x_2;r+\frac{\varepsilon}{2})>(1-s)\diamond (1-s)>1-\lambda$. Therefore, \begin{equation}\label{eqn3.1}
        \vartheta(x_1-x_2;2r+\varepsilon)>1-\lambda
    \end{equation}
    Now, if we choose $\varepsilon=mr-2r, m>2$, in Equation \ref{eqn3.1} then we have $\vartheta(x_1-x_2;mr)>1-\lambda, m>2$, which is a contradiction. This completes the proof.
\end{proof}

\begin{definition}(c.f. \cite{Karakus})
    Let  $\{x_{mn}\}$ be a double sequence in a PNS $(X,\vartheta,\diamond)$. Then a point $\xi\in X$ is said to be  statistical cluster point of $\{x_{mn}\}$ with respect to the probabilistic norm $\vartheta$ if for every $\varepsilon>0$ and  $\lambda\in(0,1)$, $\delta_2(\{(m,n)\in\mathbb{N}\times \mathbb{N}: \vartheta (x_{mn}-\xi;\varepsilon)>1-\lambda\})>0$.
\end{definition}
We denote $\Lambda_{(x_{mn})}(st_2^\vartheta)$ to mean ordinary statistical cluster points of $\{x_{mn}\}$ with respect to the probabilistic norm $\vartheta$.

\begin{definition}
    Let  $\{x_{mn}\}$ be a double sequence in a PNS $(X,\vartheta,\diamond)$. Then a point $\xi\in X$ is said to be rough statistical cluster point of $\{x_{mn}\}$ with respect to the probabilistic norm $\vartheta$ if for every $\varepsilon>0$, $\lambda\in(0,1)$ and some $r>0$, $\delta_2(\{(m,n)\in\mathbb{N}\times \mathbb{N}: \vartheta (x_{mn}-\xi;r+\varepsilon)>1-\lambda\})>0$. The set of all rough statistical cluster points of $\{x_{mn}\}$ is denoted as $\Lambda_{(x_{mn})}^r(st_2^\vartheta)$.
\end{definition}

 \begin{remark}
     If $r=0$, then $\Lambda_{(x_{mn})}^r(st_2^\vartheta)= \Lambda_{(x_{mn})}(st_2^\vartheta)$.
 \end{remark}

\begin{theorem}
   Let  $\{x_{mn}\}$ be a double sequence in a PNS $(X,\vartheta,\diamond)$. Then, the set  $\Lambda_{(x_{mn})}^r(st_2^\vartheta)$ is closed for some $r>0$.
\end{theorem}

\begin{proof}
    If $\Lambda_{(x_{mn})}^r(st_2^\vartheta)=\emptyset$ then we have nothing to prove. So, let $\Lambda_{(x_{mn})}^r(st_2^\vartheta)\neq \emptyset$. Suppose that $\{\omega_{mn}\}$ is a double sequence in $\Lambda_{(x_{mn})}^r(st_2^\vartheta)$ such that $\omega_{mn}\xrightarrow{\vartheta_2}\zeta$. Now for given $\lambda\in(0,1)$, choose $s\in (0,1)$ such that $(1-s)\diamond (1-s)>1-\lambda$. Then for every $\varepsilon>0$ there exists $k_\varepsilon\in\mathbb{N}$ such that $\vartheta(\omega_{mn}-\zeta;\frac{\varepsilon}{2})>1-s$ for all $m,n\geq k_\varepsilon$. Now choose $m_0,n_0\in\mathbb{N}$ such that $m_0,n_0>k_\varepsilon$. Then $\vartheta(\omega_{m_0n_0}-\zeta;\frac{\varepsilon}{2})>1-s$. Again, since $\{\omega_{m_0n_0}\}\in \Lambda_{(x_{mn})}^r(st_2^\vartheta)$, then $\delta_2(K)>0$ where $K=\{ (m,n)\in\mathbb{N}\times\mathbb{N}: \vartheta(x_{mn}-\omega_{m_0n_0};r+\frac{\varepsilon}{2})>1-s\}$. Now, for $(i,j)\in K$, we have $\vartheta(x_{ij}-\zeta;r+\varepsilon)\geq \vartheta(x_{ij}-\omega_{m_0n_0};r+\frac{\varepsilon}{2})\diamond \vartheta(\omega_{m_0n_0}-\zeta;\frac{\varepsilon}{2})>(1-s)\diamond (1-s)>1-\lambda$. Therefore $\{ (i,j)\in\mathbb{N}\times\mathbb{N}:  \vartheta(x_{ij}-\zeta;r+\varepsilon)>1-\lambda \}\subset K$. Since $\delta_2(K)>0$, $\delta_2(\{ (i,j)\in\mathbb{N}\times\mathbb{N}:  \vartheta(x_{ij}-\zeta;r+\varepsilon)>1-\lambda \})>0$. So, $\zeta\in \Lambda_{(x_{mn})}^r(st_2^\vartheta)$. Hence $\Lambda_{(x_{mn})}^r(st_2^\vartheta)$ is closed.
\end{proof}

\begin{theorem}\label{thm3.9}
     Let  $\{x_{mn}\}$ be a double sequence in a PNS $(X,\vartheta,\diamond)$ and $r>0$. Then, for an arbitrary $\zeta\in \Lambda_{(x_{mn})}(st_2^\vartheta)$ and $\lambda\in(0,1)$, $\vartheta(\gamma-\zeta;r)>1-\lambda$ for all $\gamma\in \Lambda_{(x_{mn})}^r(st_2^\vartheta)$.
\end{theorem}

\begin{proof}
    For given $\lambda\in(0,1)$, choose $s\in(0,1)$ such that $(1-s)\diamond (1-s)>1-\lambda$. Since $\zeta\in \Lambda_{(x_{mn})}(st_2^\vartheta)$, then for every $\varepsilon>0$, $\delta_2(M)>0$ where $M=\{(m,n)\in\mathbb{N}\times\mathbb{N}: \vartheta(x_{mn}-\zeta;\varepsilon)>1-s\}$. Now, we will show that if $\gamma\in X$ satisfying $\vartheta(\gamma-\zeta;r)>1-s$ then $\gamma\in \Lambda_{(x_{mn})}^r(st_2^\vartheta)$. Now, for $(i,j)\in M$ we have $\vartheta(x_{ij}-\gamma;r+\varepsilon)\geq \vartheta(x_{ij}-\zeta;\varepsilon)\diamond \vartheta(\gamma-\zeta;r)>(1-s)\diamond (1-s)>1-\lambda$.
    Therefore $M\subset \{(i,j)\in\mathbb{N}\times\mathbb{N}: \vartheta(x_{ij}-\gamma;r+\varepsilon)>1-\lambda\}$. Hence $\delta_2(\{(i,j)\in\mathbb{N}\times\mathbb{N}: \vartheta(x_{ij}-\gamma;r+\varepsilon)>1-\lambda\})>0$, Consequently, $\gamma\in \Lambda_{(x_{mn})}^r(st_2^\vartheta)$. This completes the proof.
\end{proof}

\begin{theorem}\label{thm3.10}
    Let  $\{x_{mn}\}$ be a double sequence in a PNS $(X,\vartheta,\diamond)$ and $r>0$. Then for  $\lambda\in (0,1)$ and fixed $x_0\in X$, $\Lambda_{(x_{mn})}^r(st_2^\vartheta)=\bigcup_{x_0\in  \Lambda_{(x_{mn})}(st_2^\vartheta)}\overline{B(x_0,\lambda,r)}$.
\end{theorem}

\begin{proof}
    For $\lambda\in(0,1)$, choose $s\in(0,1)$ such that $(1-s)\diamond (1-s)>1-\lambda$. Let $\gamma\in \bigcup_{x_0\in  \Lambda_{(x_{mn})}(st_2^\vartheta)}\overline{B(x_0,\lambda,r)}$. Then there exists a $x_0\in \Lambda_{(x_{mn})}$ such that $\vartheta(x_0-\gamma;r)>1-s$. Since $x_0\in \Lambda_{(x_{mn})}$, then for every $\varepsilon>0$, $\delta_2(Z)>0$ where $Z=\{(m,n)\in\mathbb{N}\times\mathbb{N}: \vartheta(x_{mn}-x_0;\varepsilon)>1-s\}$. Now for $(i,j)\in Z$, we have $\vartheta(x_{ij}-\gamma;r+\varepsilon)\geq \vartheta(x_{ij}-x_0;\varepsilon)\diamond \vartheta(x_0-\gamma;r)>(1-s)\diamond (1-s)>1-\lambda$. Therefore $Z\subset \{(i,j)\in\mathbb{N}\times \mathbb{N}: \vartheta(x_{ij}-\gamma;r+\varepsilon)>1-\lambda \}$. Since $\delta_2(Z)>0$, $\delta_2(\{(i,j)\in\mathbb{N}\times \mathbb{N}: \vartheta(x_{ij}-\gamma;r+\varepsilon)>1-\lambda \})>0$. Hence $\gamma\in \Lambda_{(x_{mn})}^r(st_2^\vartheta)$ and so, $\bigcup_{x_0\in  \Lambda_{(x_{mn})}(st_2^\vartheta)}\overline{B(x_0,\lambda,r)}\subset \Lambda_{(x_{mn})}^r(st_2^\vartheta)$. 

    conversely suppose that $\gamma\in \Lambda_{(x_{mn})}^r(st_2^\vartheta)$. Now, we show that $\gamma\in \bigcup_{x_0\in  \Lambda_{(x_{mn})}(st_2^\vartheta)}\overline{B(x_0,\lambda,r)}$. If possible, let $\gamma \notin  \bigcup_{x_0\in  \Lambda_{(x_{mn})}(st_2^\vartheta)}\overline{B(x_0,\lambda,r)}$. Then for every $x_0\in  \Lambda_{(x_{mn})}(st_2^\vartheta)$, $\vartheta(\gamma-x_0;r)<1-\lambda$, which contradicts the fact of Theorem \ref{thm3.9}. Hence $\gamma\in \bigcup_{x_0\in  \Lambda_{(x_{mn})}(st_2^\vartheta)}\overline{B(x_0,\lambda,r)}$. Therefore $\Lambda_{(x_{mn})}^r(st_2^\vartheta)\subset \bigcup_{x_0\in  \Lambda_{(x_{mn})}(st_2^\vartheta)}\overline{B(x_0,\lambda,r)}$. This completes the proof.
\end{proof}

\begin{theorem}\label{thm3.11}
   Let  $\{x_{mn}\}$ be a double sequence in a PNS $(X,\vartheta,\diamond)$. Then for some $r>0$ and  any $\lambda\in(0,1)$, the following statements hold:
   \begin{enumerate}
       \item If $x_0\in  \Lambda_{(x_{mn})}(st_2^\vartheta)$, then $st_2^\vartheta\text{-}LIM_{x_{mn}}^r\subseteq \overline{B(x_0,\lambda,r)}$.
       \item $st_2^\vartheta\text{-}LIM_{x_{mn}}^r=\bigcap_{x_0\in  \Lambda_{(x_{mn})}(st_2^\vartheta)}\overline{B(x_0,\lambda,r)}=\{y_0\in X:  \Lambda_{(x_{mn})}(st_2^\vartheta)\subseteq \overline{B(y_0,\lambda,r)} \}$
   \end{enumerate}
\end{theorem}

\begin{proof}
    Suppose $\{x_{mn}\}$ is a double sequence in a PNS $(X,\vartheta,\diamond)$ and $r>0$.
    \begin{enumerate}
        \item   Now, for a given $\lambda\in(0,1)$, choose $s\in(0,1)$ such that $(1-s)\diamond(1-s)>1-\lambda$. Let $\beta \in st_2^\vartheta\text{-}LIM_{x_{mn}}^r$. Then for every $\varepsilon>0$, we have $\delta_2(K_1)=0$ and $\delta_2(k_2)>0$ where $K_1=\{(m,n)\in\mathbb{N}\times \mathbb{N}: \vartheta(x_{mn}-\beta;r+\varepsilon)\leq 1-s\}$ and $K_2=\{(m,n)\in\mathbb{N}\times \mathbb{N}: \vartheta(x_{mn}-x_0;\varepsilon)>1-s \}$. Now for $(i,j)\in K_1^c\cap K_2$, we have $\vartheta(\beta-x_0;r)\geq \vartheta(x_{ij}-\beta;r+\varepsilon)\diamond \vartheta(x_{ij}-x_0;\varepsilon)>(1-s)\diamond (1-s)>1-\lambda$. Therefore $\beta\in \overline{B(x_0,\lambda,r)}$. Hence $st_2^\vartheta\text{-}LIM_{x_{mn}}^r\subseteq \overline{B(x_0,\lambda,r)}$.
        \item  By the previous part, we have $st_2^\vartheta\text{-}LIM_{x_{mn}}^r\subseteq \bigcap_{x_0\in  \Lambda_{(x_{mn})}(st_2^\vartheta)}\overline{B(x_0,\lambda,r)}$. Let $\gamma \in \bigcap_{x_0\in  \Lambda_{(x_{mn})}(st_2^\vartheta)}\overline{B(x_0,\lambda,r)}$. Then, $\vartheta(\gamma-x_0;r)\geq 1-\lambda$ for all $x_0\in  \Lambda_{(x_{mn})}(st_2^\vartheta)$ and hence $\Lambda_{(x_{mn})}(st_2^\vartheta)\subseteq \overline{B(\gamma,\lambda,r)}$, i.e., $\bigcap_{x_0\in  \Lambda_{(x_{mn})}(st_2^\vartheta)}\overline{B(x_0,\lambda,r)}\subseteq \{y_0\in X:  \Lambda_{(x_{mn})}(st_2^\vartheta)\subseteq \overline{B(y_0,\lambda,r)} \}$. Further, suppose that $\gamma\notin st_2^\vartheta\text{-}LIM_{x_{mn}}^r$. Then for $\varepsilon>0$, $\delta_2(\{(m,n)\in\mathbb{N}\times \mathbb{N}:\vartheta(x_{mn}-\gamma;r+\varepsilon)\leq 1-\lambda\})\neq 0$, which gives that there exists a statistical cluster point $x_0$ of $\{x_{mn}\}$ for which $\vartheta(\gamma-x_0;r+\varepsilon)\leq 1-\lambda$. Therefore $\Lambda_{(x_{mn})}(st_2^\vartheta)\nsubseteq \overline{B(\gamma,\lambda,r)}$ and $\gamma\notin \{y_0\in X:  \Lambda_{(x_{mn})}(st_2^\vartheta)\subseteq \overline{B(y_0,\lambda,r)} \}$. Therefore $\{y_0\in X:  \Lambda_{(x_{mn})}(st_2^\vartheta)\subseteq \overline{B(y_0,\lambda,r)} \}\subseteq st_2^\vartheta\text{-}LIM_{x_{mn}}^r$. Therefore, $st_2^\vartheta\text{-}LIM_{x_{mn}}^r=\bigcap_{x_0\in  \Lambda_{(x_{mn})}(st_2^\vartheta)}\overline{B(x_0,\lambda,r)}=\{y_0\in X:  \Lambda_{(x_{mn})}(st_2^\vartheta)\subseteq \overline{B(y_0,\lambda,r)} \}$. This completes the proof.
    \end{enumerate}
\end{proof}

\begin{theorem}\label{thm3.12}
     Let  $\{x_{mn}\}$ be a double sequence in a PNS $(X,\vartheta,\diamond)$ such that $x_{mn}\xrightarrow{st_2^\vartheta}\zeta$. Then, there exists $\lambda\in(0,1)$ such that $st_2^\vartheta\text{-}LIM_{x_{mn}}^r=\overline{B(\zeta,\lambda,r)}$ for some $r>0$.
\end{theorem}

\begin{proof}
    For given $\lambda\in(0,1)$, choose $s\in (0,1)$ such that $(1-s)\diamond (1-s)>1-\lambda$. Since $x_{mn}\xrightarrow{st_2^\vartheta}\zeta$, then for every $\varepsilon>0$, the set $Y=\{(m,n)\in\mathbb{N}\times\mathbb{N}: \vartheta(x_{mn}-\zeta;\varepsilon)\leq 1-s\}$ has the double natural density zero. Now, let $y_*\in \overline{B(\zeta,s ,r)}$. Then $\vartheta(y_*-\zeta;r)\geq 1-s$. Now, for $(i,j)\in Y^c$, $\vartheta(x_{ij}-y_*;r+\varepsilon)\geq \vartheta(x_{ij}-\zeta;
    \varepsilon)\diamond \vartheta(\zeta-y_*;r)>(1-s)\diamond (1-s)>1-\lambda$. Therefore $\{(i,j)\in\mathbb{N}\times\mathbb{N}: \vartheta(x_{ij}-y_*;r+\varepsilon)\leq 1-\lambda \}\subseteq Y$. Since $\delta_2(Y)=0$, $\delta_2(\{(i,j)\in\mathbb{N}\times\mathbb{N}: \vartheta(x_{ij}-y_*;r+\varepsilon)\leq 1-\lambda \})=0$. Consequently, $y_*\in st_2^\vartheta\text{-}LIM_{x_{mn}}^r$. Hence $\overline{B(\zeta,\lambda ,r)}\subseteq st_2^\vartheta\text{-}LIM_{x_{mn}}^r$. Again, since $x_{mn}\xrightarrow{st_2^\vartheta}\zeta$, $\Lambda_{(x_{mn})}(st_2^\vartheta)=\{\zeta\}$ and consequently, from Theorem \ref{thm3.11} we have $st_2^\vartheta\text{-}LIM_{x_{mn}}^r\subseteq \overline{B(\zeta,\lambda ,r)}$. Hence $st_2^\vartheta\text{-}LIM_{x_{mn}}^r=\overline{B(\zeta,\lambda,r)}$. This completes the proof.
. \end{proof}

\begin{theorem}
    Let  $\{x_{mn}\}$ be a double sequence in a PNS $(X,\vartheta,\diamond)$ such that $x_{mn}\xrightarrow{st_2^\vartheta}\eta$. Then $\Lambda_{(x_{mn})}^r(st_2^\vartheta)=st_2^\vartheta\text{-}LIM_{x_{mn}}^r$ for some $r>0$.
\end{theorem}

\begin{proof}
    Since $x_{mn}\xrightarrow{st_2^\vartheta}\eta$, $\Lambda_{(x_{mn})}(st_2^\vartheta)=\{\eta\}$. By Theorem \ref{thm3.10}, for $\lambda\in(0,1)$,
 $\Lambda_{(x_{mn})}^r(st_2^\vartheta)=\overline{B(\eta,\lambda,r)}$. Again, from Theorem \ref{thm3.12}, $\overline{B(\eta,\lambda,r)}=st_2^\vartheta\text{-}LIM_{x_{mn}}^r$. Therefore, $\Lambda_{(x_{mn})}^r(st_2^\vartheta)=st_2^\vartheta\text{-}LIM_{x_{mn}}^r$. This completes the proof.
\end{proof}

\subsection*{Authors' contributions}
The authors have contributed equally and significantly in writing this paper. Both the authors read and approved the final version of the manuscript.

\subsection*{Competing interests}
The authors declare that they have no competing interests.

\subsection*{Acknowledgments}
 The second author is grateful to The Council of Scientific and Industrial Research (CSIR), HRDG, India, for the grant of Senior Research Fellowship during the preparation of this paper.


\begin{thebibliography}{99}\baselineskip=20pt
\footnotesize{
\bibitem{Alsina}
C. Alsina, B. Schweizer, A. Sklar,   On the definition of a probabilistic normed space,  \textit{Aequ. Math.}, \textbf{46} (1993), 91-98.

\bibitem{Aghajani}
A. Aghajani, K. Nourouzi,   Convex sets in probabilistic normed spaces, \textit{Chaos Solitons  Fractals}, \textbf{36 (2)} (2008), 322-328.

\bibitem{Ayter}
S. Aytar,   Rough statistical convergence, \textit{Numer. Funct. Anal. Optim.}, \textbf{29(3-4)} (2008), 291-303.

\bibitem{Antal2021}
R. Antal, M. Chawla, V. Kumar,   Rough statistical convergence in intuitionistic fuzzy normed spaces, \textit{Filomat}, \textbf{35(13)}  (2021), 4405-4416.

\bibitem{Antal 2022}
R. Antal, M. Chawla, V. Kumar,  Rough statistical convergence in probabilistic normed spaces, \textit{Thai J.  Math.},  \textbf{20 (4)}, (2022), 1707-1719.

\bibitem{Connor}
J. S. Connor,   The statistical and strong p-Ces$\grave{a}$ro convergence of sequences,  \textit{Analysis}, \textbf{8 (1-2)} (1988), 47-63.

\bibitem{Cakalli}
H. \c{C}akall\i, E. Sava\c{s}, Statistical convergence of double sequences in topological groups, \textit{Journal of Computational Analysis and Applications}, \textbf{12 (1A)} (2010), 421-426.

\bibitem{Dundar}
F. Nuray, E. D\"{u}ndar,  U. Ulusu,   Wijsman statistical convergence of double sequences of sets, \textit{Iran. J. Math. Sci.  Inform.}  \textbf{16 (1)} (2021), 55-64.


\bibitem{Fast}
H. Fast,  Sur la convergence statistique,  \textit{ Colloq. Math.} \textbf{ 2 (3-4)}, (1951), 241-244.

\bibitem{Frank}
M. J. Frank,  Probabilistic topological spaces, \textit{J.  Math.  Anal.  Appl.}  \textbf{34 (1)}  (1971), 67-81.

\bibitem{Fridy}
J. A. Fridy,   On statistical convergence, \textit{Analysis}, \textbf{5(4)} (1985), 301-313.

\bibitem{Fridy1993}
J. A. Fridy, C. Orhan,   Lacunary statistical convergence, \textit{Pacific J. Math.} \textbf{160 (1)}, (1993), 43-51.

\bibitem{Ghosal}
S. Ghosal, M. Banerjee,   Effects on rough $\mathcal{I}$-lacunary statistical convergence to induce the weighted sequence, \textit{Filomat}, \textbf{32 (10)} (2018), 3557-3568.

\bibitem{Hossain}
N. Hossain, A. K.  Banerjee,   Rough $\mathcal{I}$-convergence in intuitionistic  fuzzy normed spaces, \textit{Bull.  Math. Anal.  Appl.} 14(4)  (2022), 1-10.


\bibitem{Klement}
E. P. Klement, R. Mesiar, E. Pap, Triangular norms. Position paper I: basic analytical and algebraic properties, \textit{Fuzzy Sets and Systems}, \textbf{143} (2004), 5-26.

\bibitem{Karakus}
S. Karakus,   Statistical convergence on probalistic normed spaces, \textit{Math. Commun.}, \textbf{12(1)} (2007), 11-23.

\bibitem{Karakus Demirci}
S. Karakus, K. Dem\i rc\i,   Statistical convergence of double sequences on probabilistic normed spaces,  \textit{Int. J. Math.  Math. Sci.}, Volume 2007, Article ID 14737, 11 pages.

\bibitem{Kisi}
\"{O}. Ki\c{s}i, H. K.  \"{U}nal,   Rough statistical convergence of difference double sequences in normed linear spaces, \textit{Honam Math. J.}, \textbf{43 (1)} (2021), 47-58.

\bibitem{Menger}
K. Menger,  Statistical metrices, \textit{Proc. Nat. Acad. Sci. USA}, \textbf{28 (12)} (1942), 535-537.

\bibitem{Mursaleen2000}
M. Mursaleen,  $\lambda $-statistical convergence, \textit{Math. slovaca}, \textbf{50 (1)}, (2000), 111-115.

\bibitem{Mursaleen2003}
M. Mursaleen, O. H. H. Edely, Statistical convergence of double sequences, \textit{J. Math. Anal. Appl.} \textbf{288 (1)} (2003), 223-231.

\bibitem{Mursaleen2009}
M. Mursaleen, S. A.  Mohiuddine,   Statistical convergence of double sequences in intuitionistic fuzzy normed spaces, \textit{Chaos Solitons  Fractals}, \textbf{41} 
 (2009), 2414-2421.

 \bibitem{Malik2013}
P.  Malik, M. Maity,   On rough convergence of double sequence in normed linear spaces, \textit{Bull. Allah. Math. Soc.}  \textbf{28 (1)} (2013), 89-99.

\bibitem{Malik2016}
P. Malik, M. Maity,   On rough statistical convergence of double sequences in normed linear spaces,  \textit{Afrika Matematika}, \textbf{27 (1-2)} (2016), 141-148.

\bibitem{Nuray}
F. Nuray, E. Sava\c{s},   Statistical convergence of sequences of fuzzy numbers, \textit{ Math. Slovaca}, \textbf{45 (3)}  (1995), 269-273.

\bibitem{Ozcan}
A. \"{O}zcan,  A. Or, Rough statistical convergence of double sequences in intuitionistic fuzzy normed spaces, \textit{Journal of New Results in Science}, 
 \textbf{11 (3)} (2022), 233-246.


\bibitem{Phu2001}
H. X. Phu,   Rough convergence in normed linear spaces,  \textit{Numer. Funct. Anal. Optim.},  \textbf{22(1-2)} (2001), 199-222.

\bibitem{Phu2002}
H. X. Phu, Rough continuity of linear operators, \textit{Numer. Funct. Anal. Optim.} \textbf{23}, (2002), 139-146. 

\bibitem{Phu2003}
H. X. Phu, Rough convergence in infinite dimensional normed spaces, \textit{Numer. Funct. Anal. Optim.} \textbf{24},
(2003), 285-301. 

\bibitem{Steinhaus}
H. Steinhaus,   Sur la convergence ordinaire et la convergence asymptotique, \textit{ Colloq. Math.} \textbf{ 2 (1)} (1951),  73-74.

\bibitem{Serstnev}
A. N.  \v{S}erstnev,   Random normed spaces: Problems of completeness, \textit{Kazan. Gos. Univ. Ucen. Zap.},  \textbf{122}  (1962), 3-20.

\bibitem{Salat}
T. \v{S}al\'{a}t,   On statistically convergent sequences of real numbers, \textit{Math. Slovaca}, \textbf{30 (2)} (1980), 139-150.

\bibitem{Schweizer}
B. Schweizer, A. Sklar,   Probabilistic metric spaces, \textit{North Holland, New York-Amsterdam-Oxford}, 1983.

\bibitem{Sarabadan}
S. Sarabadan, S. Talebi,   Statistical convergence of double sequences in $2$-normed spaces, \textit{Int. J. Contemp. Math. Sciences}, \textbf{6 (8)} (2011), 373-380.
}
\end{thebibliography}
\end{document}